\documentclass[11pt]{amsart}
\usepackage[lmargin=1in,rmargin=1in,tmargin=1in,bmargin=1in]{geometry}
\usepackage[ps,all,arc,rotate]{xy}
\usepackage{graphicx, float, epstopdf}
\usepackage{bbm}
\usepackage{color}
\usepackage[unicode,bookmarks=false]{hyperref}
\hypersetup{hidelinks}
\usepackage{centernot}
\usepackage{fancyhdr}
\usepackage{multirow}
\usepackage[utf8]{inputenc}
\usepackage{amsfonts,amssymb,amsmath,amsthm,mathrsfs}
\usepackage{graphics, setspace}
\usepackage{braket}
\usepackage{mathtools}
\usepackage{tikz}
\usepackage{upgreek}
\usepackage{xcolor}
\usepackage{array,esint}
\usepackage{booktabs}
\usepackage{vcell}

\numberwithin{equation}{section}
\numberwithin{figure}{section}
\allowdisplaybreaks[4]

\newtheorem{lemma}{Lemma}[section]
\newtheorem{theorem}{Theorem}[section]
\newtheorem{proposition}{Proposition}[section]

\newtheorem{corollary}[lemma]{Corollary}
\theoremstyle{definition}

\newtheorem{remark}{Remark}[section]

\newcommand{\Z}{\mathbb{Z}}
\newcommand{\R}{\mathbb{R}}
\newcommand{\N}{\mathbb{N}}
\newcommand{\e}{\operatorname{e}}
\newcommand{\eps}{\varepsilon}
\newcommand{\Major}{\mathfrak{M}}
\newcommand{\Minor}{\mathfrak{m}}

\begin{document}

\title[Exponential sum over squarefree integers]{On the exponential sum over squarefree integers}

\author[N.~Robles]{Nicolas Robles}
\address{Nicolas Robles: RAND Corporation, Engineering and Applied Sciences,
1200 S Hayes, Arlington, VA 22202, USA}
\email{robles.nicolas.m@gmail.com}

\author[A.~Zaharescu]{Alexandru Zaharescu}
\address{Alexandru Zaharescu: Department of Mathematics, University of
Illinois Urbana-Champaign, Altgeld Hall, 1409 W. Green Street, Urbana, IL,
61801, USA}
\email{zaharesc@illinois.edu}

\author[D.~Zeindler]{Dirk Zeindler}
\address{
Dirk Zeindler: Department of Mathematics and Statistics, Lancaster University, Fylde College, Bailrigg, Lancaster LA1 4YF, United Kingdom}
\email{d.zeindler@lancaster.ac.uk}

\begin{abstract}
Let $\mu$ be the M\"obius function and $\e(t)=e^{2\pi it}$.  We prove
that if $N\ge2$, $\alpha\in\R$, $(a,q)=1$, and $|\alpha-a/q|\le q^{-2}$,
then
\[
  \bigg|\sum_{n\le N}\mu^2(n)\e(\alpha n)\bigg|\ll\left(\frac Nq+q\right)(\log 2N)^5,
\]
with an absolute implied constant, and we deduce the corresponding
estimate on the minor arcs of the Hardy--Littlewood dissection throughout
the range $Q\le N^{1/2}$.  The estimates of Schlage-Puchta \cite{SP} and
of Tolev \cite{Tolev} have the same dependence on $q$ and $Q$ but carry a
factor $N^{\eps}$.  The proof uses Heath-Brown's square sieve with sieving primes confined to
an interval $(P,2P]$, where $P$ may be as small as a multiple of $\log N$;
a finite Fej\'er majorant in place of a truncated Fourier series; and,
after completion of the character sums, a count of representations that
exploits the restriction on the primes in place of the divisor function.
\end{abstract}

\subjclass[2020]{Primary: 11L07. Secondary: 11N36, 11P55}
\keywords{squarefree integers, exponential sums, square sieve, quadratic characters, Diophantine approximation}
\maketitle

\section{Introduction}

Let $\mu$ denote the M\"obius function, so that $\mu^2$ is the indicator
function of the squarefree positive integers, and write $\e(t)=e^{2\pi it}$.
For real $\alpha$ and $N\ge2$ we consider the exponential sum
\begin{equation}\label{eq:S-def}
  S(\alpha)=S(\alpha;N):=\sum_{n\le N}\mu^2(n)\e(\alpha n).
\end{equation}
This sum is the generating function of the squarefree integers in the
Hardy--Littlewood circle method, and it governs the additive problems in
which the summands are squarefree.  If $\nu\in\N$ and
\begin{equation}\label{eq:r-def}
  r_\nu(N):=\sum_{\substack{n_1+\cdots+n_\nu=N\\ n_1,\ldots,n_\nu\ge1}}
  \mu^2(n_1)\cdots\mu^2(n_\nu)
\end{equation}
denotes the number of representations of $N$ as a sum of $\nu$ squarefree
positive integers, then, for integral $N$,
\begin{equation}\label{eq:orthogonality}
  r_\nu(N)=\int_0^1S(\alpha)^\nu\e(-\alpha N)\,d\alpha,
\end{equation}
and asymptotic
formulae for $r_\nu(N)$ were obtained by Evelyn and Linfoot~\cite{EL} and
sharpened by Mirsky~\cite{Mirsky}.  The modern treatment is due to
Br\"udern and Perelli~\cite{BP}, who combined the mean square estimate of
Br\"udern, Granville, Perelli, Vaughan, and Wooley~\cite{BGPVW} for
$S(\alpha)$ on the minor arcs with a pointwise estimate on the same arcs.

To state that pointwise estimate, let $1\le Q\le N^{1/2}$ and let
$\Major(Q)$ be the union of the intervals
\begin{equation}\label{eq:major-arcs}
  \left\{\alpha\in[0,1]:\ |r\alpha-b|\le\frac QN\right\},
  \qquad 1\le r\le Q,\quad b\in\Z,\quad (b,r)=1,
\end{equation}
and put $\Minor(Q)=[0,1]\setminus\Major(Q)$.  Br\"udern and
Perelli~\cite[Theorem~4]{BP} proved that
\begin{equation}\label{eq:BP-minor}
  \sup_{\alpha\in\Minor(Q)}|S(\alpha)|\ll N^{1+\eps}Q^{-1}
\end{equation}
for $Q\le N^{3/7}$; a similar estimate had been obtained by Baker,
Br\"udern, and Harman~\cite{BBH} for $Q\le N^{1/3}$.  Here and below
$\eps$ denotes an arbitrarily small positive number, and the implied
constants in $\ll$ and $O$ may depend on it; $A\asymp B$ means that
$A\ll B\ll A$, and $\|t\|$ denotes the distance from $t$ to the nearest
integer.  Further notation is collected at the beginning of
Section~\ref{sec:lemmas}.  Br\"udern and Perelli observed \cite[p.~610]{BP} that the
restriction $Q\le N^{3/7}$ is the only obstacle to the optimal conditional
error term for $r_3(N)$, and Br\"udern conjectured (see \cite{SP}) that
\eqref{eq:BP-minor} holds in the full range $Q\le N^{1/2}$.  This was
established independently by
Schlage-Puchta~\cite{SP}, by an elementary counting argument, and by
Tolev~\cite{Tolev}, by means of Heath-Brown's square sieve~\cite{HB}.
Schlage-Puchta's estimate is in fact stated for a rational approximation
rather than for minor arcs: if $|q\alpha-a|\le q^{-1}$, then
\begin{equation}\label{eq:SP}
  |S(\alpha)|\ll N^{1+\eps}q^{-1}+N^{\eps}q,
\end{equation}
from which \eqref{eq:BP-minor} follows for $Q\le N^{1/2}$ by Dirichlet's
theorem.  The condition $(a,q)=1$ is not stated in \cite{SP}, but it is used
in the proof and cannot be dispensed with; see Remark~\ref{rem:coprime} below.

The purpose of this paper is to replace the factor $N^{\eps}$ in
\eqref{eq:SP} and \eqref{eq:BP-minor} by a fixed power of $\log N$.

\begin{theorem}\label{thm:main}
Let $N\ge2$, let $\alpha\in\R$, and let $a\in\Z$ and $q\in\N$ satisfy
\begin{equation}\label{eq:main-hyp}
  (a,q)=1,\qquad \bigg|\alpha-\frac aq\bigg|\le\frac1{q^2}.
\end{equation}
Then
\begin{equation}\label{eq:main-bound}
  |S(\alpha)|\ll\left(\frac Nq+q\right)(\log 2N)^{5},
\end{equation}
where the implied constant is absolute.
\end{theorem}

\begin{theorem}\label{thm:minor}
Let $N\ge2$ and $1\le Q\le N^{1/2}$.  Then
\begin{equation}\label{eq:minor-bound}
  \sup_{\alpha\in\Minor(Q)}|S(\alpha)|\ll\frac NQ(\log 2N)^{5},
\end{equation}
where the implied constant is absolute.
\end{theorem}

No restriction $q\le N$ is needed in Theorem~\ref{thm:main}, since for
$q>N$ the term $q$ exceeds the trivial estimate $|S(\alpha)|\le N$.  The
exponent $5$ has not been optimized.  The dependence $N/q+q$ on $q$ is the
same as in \eqref{eq:SP}, and it cannot be improved in a uniform estimate
of this shape; see Remark~\ref{rem:sharp}.  The contribution of the present
paper is therefore the logarithmic factor.  To the best of our knowledge,
no estimate of the form \eqref{eq:main-bound} or \eqref{eq:minor-bound}
with a fixed power of $\log N$ in place of $N^{\eps}$ has appeared before.  Estimates of this form are
preferable to those with $N^{\eps}$ in two respects: the implied constant
is absolute rather than dependent on $\eps$, and the estimate remains
useful in applications where the main term exceeds the error term by
less than a power of $N$.

It is worth explaining why Theorems~\ref{thm:main} and~\ref{thm:minor} do
not follow from the arguments of \cite{SP} or \cite{Tolev} by keeping track
of the powers of $\log N$.  The proof in \cite{SP} passes through the
intermediate estimate
$|S(\alpha)|\ll N^{1+\eps}q^{-1}+N^{1/2+\eps}+N^{\eps}q$, and the middle
term is then absorbed into the other two.  The dyadic subdivisions and the harmonic sum in his argument contribute
only powers of $\log N$.  The input that is not of this kind is the divisor
bound $\tau(n)\ll n^{\eps}$, where $\tau(n)$ is the number of divisors of
$n$, applied in his Lemma~1 to the number of
representations of an integer $n$ in the form
$d_1^2-d_2^2=(d_1-d_2)(d_1+d_2)$; and no fixed power of $\log n$ bounds
$\tau(n)$ pointwise (see Ramanujan~\cite{Ramanujan}), since for
$n=p_1\cdots p_k$, the product of the first $k$ primes, one has $\tau(n)=2^k$ while $\log n\ll k\log 2k$ by the prime
number theorem.  Our proof retains the restricted factorization
$\ell=hpp'$ described below in place of this general multiplicity.  In Tolev's argument the sieve weight is built from
primes $p$ in an interval $(P,2P]$ with $P\ge N^{\eta}$, the smooth
majorant of $\min(M,\|x\|^{-1})$ is truncated at frequencies
$|n|\le M^{1+\eps}$, and after completion of the character sums the
frequencies are grouped as $h=np$ and $h=npp'$ and counted with the divisor
function; these are the places at which his proof introduces $N^{\eps}$.

Our argument follows Tolev's in outline, with three changes.  The
majorant of $\min(M,\|x\|^{-1})$ is a finite Fej\'er polynomial of degree
at most $M/2$ with coefficients $O(\log M)$, so that no truncation is
needed and the zero frequency is the only one to be treated separately.
The sieving primes are confined to an interval $(P,2P]$ with $P$ as small
as a constant multiple of $\log N$, which suffices for the square detector
in the present application.
Finally, after completion of the character sums the number of
representations of an integer $\ell$ as $hpp'$, with $p\ne p'$ in
$(P,2P]$, is at most $(\log\ell/\log P)^2$, and this replaces the divisor
function.  The diagonal terms $p=p'$ are handled by positivity of the
majorant, which removes the need to separate the terms with $p\mid k$.

It is also instructive to compare with \cite{BP}.  After the usual
reduction to the dyadic sums
\begin{equation}\label{eq:dyadic-intro}
  \sum_{D<d\le2D}\min\left(\frac N{D^2},\frac1{\|\alpha d^2\|}\right),
  \qquad D\le Q,
\end{equation}
Br\"udern and Perelli use three arguments.  For $D^2\le Q/8$ and for
$Q/8<D^2\le N/Q$ the bounds they obtain are $N Q^{-1}\log N$, with no loss of
$N^{\eps}$; the factor $N^{\eps}$ in \eqref{eq:BP-minor} enters only in the
remaining range $N/Q\le D^2\le Q^2$, through an estimate of Harman for
Weyl sums with square exponents.  Tolev treats the subrange
$N/Q\le D^2<Q^{4/3}$ of this range by the square sieve and imports the
rest from \cite{BP}.  Our square-sieve estimate, Proposition~\ref{prop:sieve}
below, covers the whole range $N/Q\le D^2\le Q^2$ with a logarithmic loss,
and the two elementary ranges are treated as in \cite{BP}.  The proof of Theorem~\ref{thm:main} is
thus self-contained apart from the standard results quoted in
Section~\ref{sec:lemmas}: Dirichlet's theorem, the prime number theorem,
the evaluation of Gauss sums, and the P\'olya--Vinogradov inequality.

Mean values of $S(\alpha)$ are better understood than its pointwise
behaviour.  Balog and Ruzsa~\cite{BR} determined the order of magnitude of
$\int_0^1|S(\alpha)|\,d\alpha$, and Keil~\cite{Keil} obtained essentially
sharp estimates for all $L^p$ means; these results are of a different
nature and do not yield a pointwise estimate in the range of
Theorem~\ref{thm:main}.  In a different direction, Basak, Berndt, and
the second author~\cite[Section~10]{BBZ} proved that if $\theta_1,\theta_2>0$,
$\theta_1+\theta_2<1$, $v\le N^{\theta_1}$, and
$\beta=\alpha-u/v$ satisfies $|\beta|\le N^{\theta_2-1}$, then
\[
  |S(\alpha)|\le\min\left(N,\frac1{2|\beta|}\right)\frac{\tau(v)}v
  +O_{\eps}\left(N^{\frac12+\frac{\theta_1}2+\theta_2+\eps}
  +N^{\frac{3\theta_1}2+\theta_2+\eps}\right).
\]
This is sharper than \eqref{eq:SP}
near rationals with small denominators, where it gives the main term with
its dependence on $\beta$, but it is confined to that range and is
complementary to, and does not subsume, Theorem~\ref{thm:main}.  Theorem~\ref{thm:main} removes the intermediate term
$N^{8/13}(\log N)^{37/13}$ of an estimate in \cite{DRZZ} and improves the
power of $N$ in the central range $q\asymp N^{1/2}$; the logarithmic
exponent in \eqref{eq:main-bound} has not been optimized. See
\cite{DRZZ2} for an application of the estimates of \cite{DRZZ} to
partitions into $r$-full primes.

We conclude the introduction with two remarks on the shape of
Theorem~\ref{thm:main} and with an outline of the proof.  In
Theorems~\ref{thm:main} and~\ref{thm:minor} the parameter $N$ may be any
real number at least $2$; in statements involving $r_\nu(N)$ it is a
positive integer.  Both remarks use the identity
\begin{equation}\label{eq:mu2}
  \mu^2(n)=\sum_{d^2\mid n}\mu(d),
\end{equation}
which holds because both sides are multiplicative in $n$ and, at a prime
power $p^k$, the right-hand side is $1$ for $k\le1$ and $1+\mu(p)=0$ for
$k\ge2$; and the Euler product
\begin{equation}\label{eq:zeta2}
  \sum_{n\ge1}\frac{\mu(n)}{n^2}=\prod_p\left(1-\frac1{p^2}\right)
  =\frac1{\zeta(2)}=\frac6{\pi^2};
\end{equation}
see \cite[Chapter~1]{MV}.

\begin{remark}[The condition $(a,q)=1$]\label{rem:coprime}
Let $\alpha=1/2$, let $q\asymp N^{1/2}$ be even, and put $a=q/2$.  Then
$|\alpha-a/q|=0$, so that \eqref{eq:main-hyp} holds except for the
condition $(a,q)=1$, and the right-hand side of \eqref{eq:main-bound} is
$\ll N^{1/2}(\log 2N)^5$.  On the other hand,
\[
  S(1/2)=\sum_{d\le N^{1/2}}\mu(d)\sum_{m\le N/d^2}(-1)^{d^2m}.
\]
For odd $d$ the inner sum is $0$ or $-1$, and these $d$ contribute
$O(N^{1/2})$.  For even $d$ the inner sum is $\lfloor N/d^2\rfloor$, and
replacing $\lfloor N/d^2\rfloor$ by $N/d^2$ and extending the sum over
even $d$ to infinity introduces errors of $O(N^{1/2})$ in each case.
Since, by \eqref{eq:zeta2},
\[
  \sum_{2\mid d}\frac{\mu(d)}{d^2}
  =-\frac14\prod_{p>2}\left(1-\frac1{p^2}\right)=-\frac2{\pi^2},
\]
we obtain $S(1/2)=-\frac2{\pi^2}N+O(N^{1/2})$.  Thus
\eqref{eq:main-bound} fails for the unreduced fraction $a/q$.
\end{remark}

\begin{remark}[The terms $N/q$ and $q$]\label{rem:sharp}
Neither term in \eqref{eq:main-bound} can be replaced by a smaller power
of $q$ or of $N/q$ in an estimate that is uniform in $q$.

Let $q$ be even, let $N=3q/2$, and let $\alpha=1/q$.  By
\eqref{eq:zeta2} and the identity $\mu^2(n)=\sum_{d^2\mid n}\mu(d)$,
\begin{equation}\label{eq:squarefree-count}
  E(y):=\sum_{n\le y}\mu^2(n)-\frac6{\pi^2}y
  =\sum_{d\le y^{1/2}}\mu(d)\left\lfloor\frac y{d^2}\right\rfloor
   -\frac6{\pi^2}y\ll y^{1/2},
\end{equation}
the error coming from the integer parts and from the tail
$y\sum_{d>y^{1/2}}d^{-2}$.  Partial summation gives
\[
  S(1/q)-\frac6{\pi^2}\sum_{n\le N}\e(n/q)
  =E(N)\e(N/q)-\sum_{n<N}E(n)\bigl(\e((n+1)/q)-\e(n/q)\bigr)
  \ll q^{1/2},
\]
since $N\asymp q$, $E(n)\ll q^{1/2}$, and
$|\e((n+1)/q)-\e(n/q)|\ll q^{-1}$.  The sum of $\e(n/q)$ over the block
$1\le n\le q$ vanishes, while
$\bigl|\sum_{1\le n\le q/2}\e(n/q)\bigr|=1/\sin(\pi/q)\asymp q$.  Hence
$|S(1/q)|\gg q$ for all large even $q$, whereas $N/q\asymp1$.

Now let $p$ be an odd prime, let $q=p^2$, and let $\alpha=1/p^2$.
Split the identity $\mu^2(n)=\sum_{d^2\mid n}\mu(d)$ according as
$p\mid d$ or $p\nmid d$.  If $p\nmid d$, then $\|d^2/p^2\|\ge p^{-2}$, so
by Lemma~\ref{lem:geometric} the inner sum $\sum_{m\le N/d^2}\e(d^2m/p^2)$
is $O(p^2)$, and these $d$ contribute $O(p^2N^{1/2})$.  If $p\mid d$,
write $d=pm$; then $\mu(pm)=-\mu(m)$ for $(m,p)=1$ and $\mu(pm)=0$
otherwise, the inner sum equals $\lfloor N/(p^2m^2)\rfloor$, and as in
\eqref{eq:squarefree-count} these $d$ contribute
\[
  -\frac N{p^2}\sum_{\substack{m\ge1\\(m,p)=1}}\frac{\mu(m)}{m^2}
  +O(N^{1/2}/p)
  =-\frac{6}{\pi^2(p^2-1)}N+O(N^{1/2}/p),
\]
by \eqref{eq:zeta2}.  Taking $N=p^{10}$, say, gives $|S(1/q)|\gg N/q$.

The same computation shows what happens at an arbitrary fixed reduced
fraction $a/q$.  By Lemma~\ref{lem:geometric}, the inner sum in
$S(a/q)=\sum_{d\le N^{1/2}}\mu(d)\sum_{m\le N/d^2}\e(ad^2m/q)$ is
$\lfloor N/d^2\rfloor$ if $q\mid d^2$ and at most $q/2$ otherwise.
Replacing $\lfloor N/d^2\rfloor$ by $N/d^2$ costs $O(N^{1/2})$, extending
the sum over $d$ with $q\mid d^2$ to infinity costs
$N\sum_{d>N^{1/2}}d^{-2}\ll N^{1/2}$, and the remaining $d$ contribute
$O(qN^{1/2})$; hence
\[
  S(a/q)=N\sum_{q\mid d^2}\frac{\mu(d)}{d^2}+O(qN^{1/2}),
  \qquad\text{so that}\qquad
  \lim_{N\to\infty}\frac{S(a/q)}N=\sum_{q\mid d^2}\frac{\mu(d)}{d^2}.
\]
A squarefree $d$ satisfies $q\mid d^2$ if and only if $q$ is cube-free and
every prime dividing $q$ divides $d$.  Hence the limit vanishes unless
$q=q_1q_2^2$ with $q_1,q_2$ squarefree and coprime, in which case it equals
\[
  \frac{\mu(q_1q_2)}{(q_1q_2)^2}\sum_{\substack{e\ge1\\(e,q_1q_2)=1}}\frac{\mu(e)}{e^2}
  =\frac6{\pi^2}\,\mu(q_1q_2)\prod_{p\mid q_1q_2}\frac1{p^2-1},
\]
which has absolute value $\asymp(q_1q_2)^{-2}=(qq_1)^{-1}$.  Thus at the
rational points the term $N/q$ is attained when the squarefree part $q_1$
of $q$ is bounded, and for squarefree $q$ the main term is only of size
$N/q^2$.  An estimate that improves on \eqref{eq:main-bound} must therefore
take the factorization of $q$ into account.
\end{remark}

The proof of Theorem~\ref{thm:main} proceeds as follows.  In
Section~\ref{sec:lemmas} we collect the auxiliary results: the standard
estimate for $\sum\min(K,\|\alpha h\|^{-1})$, the Fej\'er majorant, the
evaluation of Gauss sums and the P\'olya--Vinogradov inequality for the
characters that occur, and the completion of a quadratic character sum
twisted by an additive character.  Section~\ref{sec:sieve} contains the
square-sieve estimate for the dyadic sums \eqref{eq:dyadic-intro} in the
range $N/Q\le D^2\le Q^2$.  In Section~\ref{sec:minor} the remaining
ranges of $D$ are treated by the trivial estimate and by the arguments of
\cite{BP}, and Theorem~\ref{thm:minor} is deduced.  Theorem~\ref{thm:main}
follows in Section~\ref{sec:rational}.  In Section~\ref{sec:three} we record
what Theorem~\ref{thm:minor} gives for $r_3(N)$ under the generalized
Riemann hypothesis, and where a factor $N^{\eps}$ nevertheless remains.

\section{Auxiliary lemmas}\label{sec:lemmas}

Throughout, $N\ge2$ is real, and we write
\begin{equation}\label{eq:L-def}
  L:=\log 2N.
\end{equation}
The letters $p,p'$ denote primes, $\|t\|$ is the distance from $t$ to the
nearest integer, $(u,v)$ is the greatest common divisor, $\tau(n)$ is the
number of divisors of $n$, and $\pi(y)$ is the number of primes not
exceeding $y$.  For complex $A$ and nonnegative $B$, the notation $A\ll B$,
or $A=O(B)$, means that $|A|\le CB$ for an absolute constant $C$; $B\gg A$
means $A\ll B$, and $A\asymp B$ means that $A\ll B$ and $B\ll A$.  For a
condition $E$, $\mathbf1_E$ is $1$ if $E$ holds and $0$ otherwise.  Sums over $d\sim D$ are over
$D<d\le2D$.  An expression $\min(Y,(2\|t\|)^{-1})$ is interpreted as $Y$
when $\|t\|=0$.  For an odd prime $p$, $(\tfrac np)$ is the Legendre
symbol, equal to $0$ if $p\mid n$, to $1$ if $p\nmid n$ and $n$ is a
quadratic residue modulo $p$, and to $-1$ otherwise; for distinct odd
primes $p,p'$ we put
\begin{equation}\label{eq:chi-def}
  \chi_{pp'}(n):=\left(\frac np\right)\left(\frac n{p'}\right),
\end{equation}
which is a real Dirichlet character modulo $pp'$; it is not principal,
since it takes the value $-1$ at any $n$ that is a quadratic nonresidue
modulo $p$ and congruent to $1$ modulo $p'$.

\begin{lemma}\label{lem:geometric}
Let $Y\ge1$ and $\theta\in\R$, and let $I$ be a set of at most $Y$
consecutive integers.  Then
\[
  \bigg|\sum_{n\in I}\e(\theta n)\bigg|\le\min\left(Y,\frac1{2\|\theta\|}\right).
\]
\end{lemma}

\begin{proof}
The trivial estimate gives $Y$.  If $\|\theta\|>0$, the sum is a
geometric progression, and since
$|1-\e(\theta)|=2|\sin\pi\theta|\ge4\|\theta\|$, its absolute value is at
most $2/|1-\e(\theta)|\le(2\|\theta\|)^{-1}$.
\end{proof}

The next lemma is the form of the standard estimate for
$\sum\min(K,\|\alpha h\|^{-1})$ that we need; compare
\cite[Lemma~2.2]{Vaughan}.

\begin{lemma}\label{lem:spacing}
Let $H,K\ge1$, let $(b,r)=1$, and suppose that $\alpha=b/r+\beta$ with
$|\beta|\le r^{-2}$.  Then
\begin{equation}\label{eq:spacing}
  \sum_{1\le h\le H}\min\left(K,\frac1{2\|\alpha h\|}\right)
  \ll\left(\frac{HK}r+H+K+r\right)\log 2r.
\end{equation}
\end{lemma}

\begin{proof}
For $r\le3$ the trivial bound $HK$ suffices, so let $r\ge4$.  Divide
$[1,H]$ into at most $2H/r+1$ blocks of consecutive integers of length at
most $r/2$.  If $h_1\ne h_2$ lie in the same block, then
$1\le|h_1-h_2|\le r/2<r$, so $\|b(h_1-h_2)/r\|\ge1/r$, and
\[
  \|\alpha(h_1-h_2)\|\ge\left\|\frac{b(h_1-h_2)}r\right\|-|\beta||h_1-h_2|
  \ge\frac1r-\frac1{2r}=\frac1{2r}.
\]
Thus the points $\alpha h$, for $h$ in one block, are spaced at least
$1/(2r)$ apart modulo $1$.  At most one of them satisfies
$\|\alpha h\|<1/(4r)$, and for each $j\ge1$ at most two of them satisfy
$j/(4r)\le\|\alpha h\|<(j+1)/(4r)$, since this set is the union of two
intervals of length $1/(4r)$.  The contribution of one block is therefore
at most
\[
  K+\sum_{1\le j\le2r}\frac{2\cdot4r}{2j}\ll K+r\log 2r,
\]
and \eqref{eq:spacing} follows on multiplying by the number of blocks.
\end{proof}

\begin{lemma}\label{lem:majorant}
Let $M\ge1$.  There is a trigonometric polynomial
\begin{equation}\label{eq:G-def}
  G_M(x)=\sum_{|h|\le H_M}c_h\e(hx),
  \qquad H_M\le\frac M2,
\end{equation}
with real coefficients satisfying $c_{-h}=c_h$ and
$0\le c_h\le12\log 2M$, such that $G_M(x)\ge0$ and
\begin{equation}\label{eq:G-majorizes}
  \min\left(M,\frac1{2\|x\|}\right)\le G_M(x)
\end{equation}
for all real $x$.
\end{lemma}

\begin{proof}
For an integer $H\ge1$, the Fej\'er kernel
\[
  F_H(x):=\frac1H\bigg|\sum_{0\le n<H}\e(nx)\bigg|^2
  =\sum_{|h|<H}\left(1-\frac{|h|}H\right)\e(hx)
  =\frac1H\left(\frac{\sin\pi Hx}{\sin\pi x}\right)^2
\]
is nonnegative with nonnegative coefficients at most $1$.  If
$0<\|x\|\le1/(4H)$, then $H\|x\|\le1/4$, so
$|\sin\pi Hx|=\sin(\pi H\|x\|)\ge2H\|x\|$, while $|\sin\pi x|\le\pi\|x\|$;
hence $F_H(x)\ge4H/\pi^2$, and the same holds at $x\in\Z$ since
$F_H(0)=H$.

If $M<8$, take $G_M=8$.  If $M\ge8$, put $J=\lceil\log_2M\rceil$ and
$H_j=2^{j-2}$ for $2\le j\le J$.  With $\delta=\|x\|$ we claim that
\begin{equation}\label{eq:dyadic-majorant}
  \min\left(M,\frac1{2\delta}\right)
  \le2+\sum_{2\le j\le J}2^{j-1}\mathbf1_{\{\delta\le2^{-j}\}}.
\end{equation}
If $\delta>1/4$, the left side is less than $2$.  If
$2^{-k-1}<\delta\le2^{-k}$ with $2\le k<J$, the right side is
$2+\sum_{j=2}^k2^{j-1}=2^k$ and the left side is less than $2^k$.  If
$\delta\le2^{-J}$, the right side is $2^J\ge M$.  Since
$1/(4H_j)=2^{-j}$, the polynomial
\[
  G_M(x):=2+\frac{\pi^2}2\sum_{2\le j\le J}F_{H_j}(x)
\]
majorizes the right side of \eqref{eq:dyadic-majorant}.  It is
nonnegative, its coefficients are nonnegative and symmetric, each is at most
$2+\frac{\pi^2}2(J-1)\le2+5\log_2M\le12\log2M$, and its degree is less
than $H_J=2^{J-2}<M/2$.
\end{proof}

\begin{lemma}[Dirichlet]\label{lem:Dirichlet}
Let $\alpha\in\R$ and $H\ge1$.  There are coprime integers $b$ and
$r\ge1$ with $r\le H$ and $|r\alpha-b|\le1/H$.
\end{lemma}

\begin{proof}
See \cite[Lemma~2.1]{Vaughan}.
\end{proof}

\begin{lemma}\label{lem:primes}
There are absolute constants $c>0$ and $y_0\ge2$ such that
$\pi(2y)-\pi(y)\ge cy/\log y$ for $y\ge y_0$.
\end{lemma}

\begin{proof}
This follows from the prime number theorem; see \cite[Theorem~6.9]{MV}.
\end{proof}

\begin{lemma}\label{lem:Gauss}
Let $p\ne p'$ be odd primes, let $m=pp'$, and for $t\in\Z$ put
\[
  G_m(t):=\sum_{u\bmod m}\chi_{pp'}(u)\e\left(\frac{tu}m\right).
\]
Then $|G_m(t)|\le\sqrt m$ for every $t$, and $G_m(t)=0$ if $(t,m)>1$.
\end{lemma}

\begin{proof}
For an odd prime $p$ and $t\in\Z$, let
$G_p(t)=\sum_{u\bmod p}(\tfrac up)\e(tu/p)$.  Then $G_p(t)=0$ if
$p\mid t$, and $G_p(t)=(\tfrac tp)G_p(1)$ with $|G_p(1)|=\sqrt p$ if
$p\nmid t$; see \cite[Theorem~9.7]{MV}, applied to the primitive real character $(\tfrac\cdot p)$.
Choose $\overline{p'}$ and $\bar p$ with $p'\overline{p'}\equiv1\pmod p$
and $p\bar p\equiv1\pmod{p'}$.  As $v$ runs modulo $p$ and $w$ runs modulo
$p'$, $u=vp'\overline{p'}+wp\bar p$ runs over a complete residue system
modulo $m$, with $u\equiv v\pmod p$ and $u\equiv w\pmod{p'}$.  Hence
$\chi_{pp'}(u)=(\tfrac vp)(\tfrac w{p'})$ and
$\e(tu/m)=\e(tv\overline{p'}/p)\e(tw\bar p/p')$, so that
\[
  G_m(t)=G_p(t\overline{p'})\,G_{p'}(t\bar p),
\]
and both assertions follow from the prime-modulus case.
\end{proof}

\begin{lemma}[P\'olya--Vinogradov]\label{lem:PV}
Let $p\ne p'$ be odd primes, $m=pp'$, and let $I$ be a set of consecutive
integers.  Then
\[
  \sum_{n\in I}\chi_{pp'}(n)\ll\sqrt m\log m.
\]
\end{lemma}

\begin{proof}
Since $\chi_{pp'}$ is a nonprincipal character modulo $m$, this is the
P\'olya--Vinogradov inequality; see \cite[Theorem~9.18]{MV}.
\end{proof}

The final lemma completes, modulo $pp'r$, a quadratic character modulo
$pp'$ multiplied by an additive character modulo $r$.

\begin{lemma}\label{lem:completion}
Let $p\ne p'$ be odd primes, $m=pp'$, and let $r\ge1$ and $b$ be integers
with $(m,r)=1$ and $(b,r)=1$.  Let $I$ be a set of consecutive integers and
let $h$ be an integer with $1\le|h|<r$.  Then
\begin{equation}\label{eq:completion}
  \sum_{k\in I}\chi_{pp'}(k)\e\left(\frac{bhk}r\right)
  \ll r\sqrt m\sum_{\substack{0<|s|\le mr/2\\ s+bhm\equiv0\ (\mathrm{mod}\ r)}}
  \frac1{|s|}.
\end{equation}
\end{lemma}

\begin{proof}
Put $R=mr$ and $f(x)=\chi_{pp'}(x)\e(bhx/r)$, which has period $R$.  For
$s\in\Z$ let $\widehat f(s)=\sum_{x\bmod R}f(x)\e(sx/R)$, so that
\begin{equation}\label{eq:inversion}
  f(k)=\frac1R\sum_{s\bmod R}\widehat f(s)\e\left(-\frac{sk}R\right)
\end{equation}
by the orthogonality of additive characters modulo $R$.  Choose $\bar r$
and $\bar m$ with $r\bar r\equiv1\pmod m$ and $m\bar m\equiv1\pmod r$.
As $u$ runs modulo $m$ and $v$ runs modulo $r$, $x=ur\bar r+vm\bar m$ runs
over a complete residue system modulo $R$ with $x\equiv u\pmod m$ and
$x\equiv v\pmod r$, whence
\[
  \widehat f(s)
  =\sum_{u\bmod m}\chi_{pp'}(u)\e\left(\frac{s\bar ru}m\right)
   \sum_{v\bmod r}\e\left(\frac{(bh+s\bar m)v}r\right)
  =G_m(s\bar r)\cdot r\,\mathbf1_{\{bh+s\bar m\equiv0\ (\mathrm{mod}\ r)\}}.
\]
The congruence $bh+s\bar m\equiv0\pmod r$ is equivalent to
$s+bhm\equiv0\pmod r$, and Lemma~\ref{lem:Gauss} gives
$|G_m(s\bar r)|\le\sqrt m$.  Since $(bm,r)=1$ and $0<|h|<r$, the
congruence is not satisfied by $s=0$.  Represent $s$ by an integer in
$(-R/2,R/2]$ and apply \eqref{eq:inversion} and Lemma~\ref{lem:geometric}:
\[
  \sum_{k\in I}f(k)
  \ll\frac1R\sum_{\substack{0<|s|\le R/2\\ s+bhm\equiv0\ (\mathrm{mod}\ r)}}
  r\sqrt m\cdot\frac R{2|s|},
\]
which is \eqref{eq:completion}.
\end{proof}

\section{The square-sieve estimate}\label{sec:sieve}

In this section $\alpha$ is given by a rational approximation
\begin{equation}\label{eq:minor-approx}
  \alpha=\frac br+\beta,\qquad (b,r)=1,\qquad Q<r\le\frac NQ,\qquad
  |\beta|\le\frac Q{rN},
\end{equation}
where
\begin{equation}\label{eq:Q-range}
  2\le Q\le\frac{N^{1/2}}{64}.
\end{equation}
This is the situation produced by Dirichlet's theorem on the minor arcs
$\Minor(Q)$; see Section~\ref{sec:minor}.  We estimate the dyadic sums
\eqref{eq:dyadic-intro} in the range
\begin{equation}\label{eq:D-range}
  \frac NQ\le D^2\le Q^2,\qquad M:=\frac N{D^2},
\end{equation}
which is the range in which \cite{BP} and \cite{Tolev} lose a factor
$N^{\eps}$.  Note that \eqref{eq:minor-approx}--\eqref{eq:D-range} give
\begin{equation}\label{eq:M-r}
  1\le M\le Q<r\le N.
\end{equation}

\begin{proposition}\label{prop:sieve}
There are absolute constants $C_0\ge1$ and $N_0\ge2$ with the following
property.  Let $N\ge N_0$, and suppose that \eqref{eq:minor-approx},
\eqref{eq:Q-range}, and \eqref{eq:D-range} hold.  Let $P$ be a real number
with
\begin{equation}\label{eq:P-range}
  C_0L\le P\le N^{1/2}.
\end{equation}
Then
\begin{equation}\label{eq:sieve-bound}
  \sum_{d\sim D}\min\left(M,\frac1{2\|\alpha d^2\|}\right)
  \ll L^4\left(\frac N{Pr}+\frac rP+\frac{D^2}P+PM\right).
\end{equation}
\end{proposition}

\begin{proof}
Let $c$ and $y_0$ be the constants of Lemma~\ref{lem:primes}.  Fix
$C_0\ge8/c$, and choose $N_0$ so that $C_0\log 2N\ge\max(y_0,3)$ and
$C_0\log2N\le N^{1/2}$ for all $N\ge N_0$.

\medskip\noindent\emph{The sieve weight.}
Let
\[
  \mathcal P:=\{p:\ P<p\le2P,\ p\nmid r\},\qquad A:=|\mathcal P|.
\]
All primes in $\mathcal P$ are odd, since $P\ge3$.  The distinct primes
exceeding $P$ that divide $r$ have product at most $r\le N$, so there are
at most $\log r/\log P\le L/\log P$ of them, and Lemma~\ref{lem:primes}
together with $P\ge C_0L\ge8L/c$ gives
\begin{equation}\label{eq:A-lower}
  A\ge c\frac P{\log P}-\frac L{\log P}\ge\frac{cP}{2\log P},
  \qquad\text{so that}\qquad
  \frac1A\ll\frac{\log P}P.
\end{equation}
Following Heath-Brown~\cite{HB} and Tolev~\cite{Tolev}, define
\begin{equation}\label{eq:kappa}
  \kappa(k):=\frac1{A^2}\bigg|\sum_{p\in\mathcal P}\left(\frac kp\right)\bigg|^2
  \qquad(k\in\Z),
\end{equation}
so that $\kappa(k)\ge0$.  If $k=d^2$ with $1\le d\le2D\le N$, then
$(\tfrac{d^2}p)=1$ unless $p\mid d$, and the primes in $\mathcal P$
dividing $d$ number at most $\log d/\log P\le L/\log P\le A/2$, by
\eqref{eq:A-lower} and $P\ge C_0L$.  Hence
\begin{equation}\label{eq:kappa-square}
  \kappa(d^2)\ge\frac14\qquad(1\le d\le2D).
\end{equation}
In \cite{Tolev} the primes are taken in a range $N^{\eta}\le P\le N^{1/2}$;
the argument above shows that $P\ge C_0L$ suffices for
\eqref{eq:kappa-square}, and the smaller admissible values of $P$ are
essential below.

\medskip\noindent\emph{The majorant.}
Let $G_M$ be the polynomial of Lemma~\ref{lem:majorant}, and write
\begin{equation}\label{eq:G-here}
  G_M(x)=\sum_{|h|\le H}c_h\e(hx),\qquad H\le\frac M2<r,\qquad
  0\le c_h=c_{-h}\le12L,
\end{equation}
where $H<r$ follows from \eqref{eq:M-r} and $c_h\le12\log2M\le12L$ from
$M\le N$.  Since $G_M\ge0$ and $\kappa\ge0$, Lemma~\ref{lem:majorant} and
\eqref{eq:kappa-square} give
\begin{equation}\label{eq:Y-to-Z}
  \sum_{d\sim D}\min\left(M,\frac1{2\|\alpha d^2\|}\right)
  \le4\sum_{d\sim D}\kappa(d^2)G_M(\alpha d^2)
  \le4\sum_{D^2<k\le4D^2}\kappa(k)G_M(\alpha k),
\end{equation}
where in the last step we have used that the squares $d^2$, $d\sim D$,
are distinct integers in $(D^2,4D^2]$.  Tolev uses instead a smooth
majorant whose Fourier series is truncated at $|h|\le M^{1+\eps}$; the
polynomial $G_M$ has exact degree at most $M/2$, and the inequality
$H<r$ is what allows the zero frequency to be the only exceptional one in
Lemma~\ref{lem:completion}.  Expanding \eqref{eq:kappa}, the right-hand
side of \eqref{eq:Y-to-Z} equals
\begin{equation}\label{eq:expanded}
  \frac4{A^2}\sum_{p,p'\in\mathcal P}\ \sum_{D^2<k\le4D^2}
  \left(\frac kp\right)\left(\frac k{p'}\right)G_M(\alpha k).
\end{equation}
We separate the terms with $p=p'$ from those with $p\ne p'$.

\medskip\noindent\emph{The diagonal.}
Since $(\tfrac kp)^2\le1$ and $G_M\ge0$, the terms with $p=p'$ contribute
at most
\[
  \frac4A\sum_{D^2<k\le4D^2}G_M(\alpha k)
\]
to \eqref{eq:expanded}.  This is the point at which we differ from
\cite{Tolev}, where the condition $p\nmid k$ is retained and the terms with
$p\mid k$ are removed by a second application of the geometric-sum
estimate, which introduces the divisor function; here the diagonal is
simply majorized.  Opening $G_M$, using $c_h=c_{-h}$, and applying
Lemma~\ref{lem:geometric} to the sum over $k$,
\begin{align*}
  \sum_{D^2<k\le4D^2}G_M(\alpha k)
  &=c_0\cdot\#\{k:\ D^2<k\le4D^2\}
   +2\,\mathrm{Re}\sum_{1\le h\le H}c_h\sum_{D^2<k\le4D^2}\e(\alpha hk)\\
  &\ll L\bigg(D^2+\sum_{1\le h\le H}\min\Big(3D^2,\frac1{2\|\alpha h\|}\Big)\bigg).
\end{align*}
By \eqref{eq:minor-approx}, $|\beta|\le Q/(rN)\le r^{-2}$ since
$r\le N/Q$, so Lemma~\ref{lem:spacing} applies with $K=3D^2$; as
$HK\le\tfrac32MD^2=\tfrac32N$, $H\le M$, and $\log 2r\le L$, it gives
\[
  \sum_{D^2<k\le4D^2}G_M(\alpha k)\ll L^2\left(D^2+\frac Nr+M+r\right).
\]
By \eqref{eq:A-lower}, $\log P\le L$, and $M\le r$, the diagonal
contribution is
\begin{equation}\label{eq:diag-final}
  \ll L^3\left(\frac N{Pr}+\frac{D^2}P+\frac rP\right).
\end{equation}

\medskip\noindent\emph{The off-diagonal terms: zero frequency.}
For $p\ne p'$ we write $\chi_{pp'}$ as in \eqref{eq:chi-def}.  The terms
$h=0$ of $G_M$ contribute
\[
  \frac{4c_0}{A^2}\sum_{\substack{p,p'\in\mathcal P\\p\ne p'}}
  \sum_{D^2<k\le4D^2}\chi_{pp'}(k)
\]
to \eqref{eq:expanded}.  By Lemma~\ref{lem:PV} with $m=pp'\le4P^2$,
each inner sum is $\ll P\log 2P\ll PL$, so this contribution is
\begin{equation}\label{eq:off-zero}
  \ll\frac L{A^2}\cdot A^2\cdot PL=PL^2\le PML^2,
\end{equation}
using $M\ge1$.

\medskip\noindent\emph{The off-diagonal terms: nonzero frequencies.}
For $p\ne p'$ and $1\le h\le H$ put
\[
  T_{p,p'}(h):=\sum_{D^2<k\le4D^2}\chi_{pp'}(k)\e(\alpha hk).
\]
Since $\chi_{pp'}$ is real, $T_{p,p'}(-h)=\overline{T_{p,p'}(h)}$, and
as $c_{-h}=c_h$ it suffices to bound
\begin{equation}\label{eq:off-nonzero-start}
  \frac8{A^2}\sum_{\substack{p,p'\in\mathcal P\\p\ne p'}}
  \sum_{1\le h\le H}c_h|T_{p,p'}(h)|.
\end{equation}
To remove $\beta$, let
$A_{p,p'}(t)=\sum_{D^2<k\le t}\chi_{pp'}(k)\e(bhk/r)$ for
$D^2\le t\le4D^2$.  Partial summation gives
\[
  T_{p,p'}(h)=A_{p,p'}(4D^2)\e(4\beta hD^2)
  -2\pi i\beta h\int_{D^2}^{4D^2}A_{p,p'}(t)\e(\beta ht)\,dt,
\]
and hence
$|T_{p,p'}(h)|\ll(1+|\beta|hD^2)\max_{D^2\le t\le4D^2}|A_{p,p'}(t)|$.
By \eqref{eq:minor-approx}, \eqref{eq:G-here}, and \eqref{eq:D-range},
\[
  |\beta|hD^2\le\frac Q{rN}\cdot\frac M2\cdot D^2=\frac Q{2r}<1.
\]
Since $p,p'\nmid r$, $(b,r)=1$, and $1\le h\le H<r$,
Lemma~\ref{lem:completion} applies to each interval $(D^2,t]$; with
$m=pp'\le4P^2$ and $\sqrt m\le2P$ it gives
\begin{equation}\label{eq:T-completed}
  |T_{p,p'}(h)|\ll rP\sum_{\substack{0<|s|\le2P^2r\\ s+bhpp'\equiv0\ (\mathrm{mod}\ r)}}\frac1{|s|}.
\end{equation}

We now count the triples $(h,p,p')$ according to the value of
$\ell=hpp'$.  Let
\[
  R(\ell):=\#\{(h,p,p'):\ 1\le h\le H,\ p,p'\in\mathcal P,\ p\ne p',\ hpp'=\ell\}.
\]
Given $\ell$ and an ordered pair $(p,p')$ of distinct primes of
$\mathcal P$ dividing $\ell$, the integer $h$ is determined.  If
$\nu_{\mathcal P}(\ell)$ denotes the number of primes of $\mathcal P$
dividing $\ell$, then the product of these primes is at least
$P^{\nu_{\mathcal P}(\ell)}$ and at most $\ell$, and therefore
\begin{equation}\label{eq:R-bound}
  R(\ell)\le\nu_{\mathcal P}(\ell)^2\le\left(\frac{\log\ell}{\log P}\right)^2.
\end{equation}
Tolev counts the same triples with the divisor function
$\tau^2(\ell)\ll\ell^{\eps}$; the restriction of $p,p'$ to $(P,2P]$ is
what makes \eqref{eq:R-bound} available.  Moreover
$\ell=hpp'\le H(2P)^2\le2P^2M\le2N^{3/2}$, so $\log\ell\ll L$.  For
fixed $s$, the condition $s+b\ell\equiv0\pmod r$ places $\ell$ in a single
residue class modulo $r$, since $(b,r)=1$, and an interval of length
$2P^2M$ contains at most $1+2P^2M/r$ integers of that class.  Hence
\begin{equation}\label{eq:ell-count}
  \sum_{\substack{\ell\le2P^2M\\ s+b\ell\equiv0\ (\mathrm{mod}\ r)}}R(\ell)
  \ll\frac{L^2}{(\log P)^2}\left(1+\frac{P^2M}r\right).
\end{equation}
Inserting \eqref{eq:T-completed} into \eqref{eq:off-nonzero-start},
interchanging the order of summation, and using \eqref{eq:ell-count},
$c_h\le12L$, and $\sum_{0<|s|\le2P^2r}|s|^{-1}\ll\log(4P^2r)\ll L$
(recall $P\le N^{1/2}$ and $r\le N$), we find that
\eqref{eq:off-nonzero-start} is
\[
  \ll\frac{LrP}{A^2}\sum_{0<|s|\le2P^2r}\frac1{|s|}
  \sum_{\substack{\ell\le2P^2M\\ s+b\ell\equiv0\ (\mathrm{mod}\ r)}}R(\ell)
  \ll\frac{LrP}{A^2}\cdot L\cdot\frac{L^2}{(\log P)^2}\left(1+\frac{P^2M}r\right)
  \ll L^4\left(\frac rP+PM\right),
\]
by \eqref{eq:A-lower}.  Together with \eqref{eq:Y-to-Z},
\eqref{eq:diag-final}, and \eqref{eq:off-zero}, this proves
\eqref{eq:sieve-bound}.
\end{proof}

\section{The minor arcs}\label{sec:minor}

We first prove a pointwise estimate under the hypothesis
\eqref{eq:minor-approx}, and then deduce Theorem~\ref{thm:minor}.

\begin{proposition}\label{prop:minor}
Let $N\ge2$, let $2\le Q\le N^{1/2}/64$, and suppose that
\eqref{eq:minor-approx} holds.  Then
\begin{equation}\label{eq:minor-prop}
  |S(\alpha)|\ll\frac NQL^5,
\end{equation}
where the implied constant is absolute.
\end{proposition}

\begin{proof}
Let $C_0$ and $N_0$ be the constants of Proposition~\ref{prop:sieve}.  If
$N<N_0$, then \eqref{eq:minor-prop} follows from the trivial estimate
$|S(\alpha)|\le N$, since $Q\le N_0^{1/2}$; so let $N\ge N_0$.  From
$\mu^2(n)=\sum_{d^2\mid n}\mu(d)$ and Lemma~\ref{lem:geometric},
\begin{equation}\label{eq:S-dyadic}
  |S(\alpha)|
  =\bigg|\sum_{d\le N^{1/2}}\mu(d)\sum_{m\le N/d^2}\e(\alpha d^2m)\bigg|
  \le\sum_{d\le N^{1/2}}\min\left(\frac N{d^2},\frac1{2\|\alpha d^2\|}\right).
\end{equation}
For $j\ge0$ let $D=2^{j-1}$; the intervals $D<d\le2D$ cover the positive
integers, the first of them containing only $d=1$, and $O(L)$ of them meet
$[1,N^{1/2}]$.  For such $D$ put
\begin{equation}\label{eq:Y-def}
  \mathcal Y(D):=\sum_{D<d\le\min(2D,N^{1/2})}
  \min\left(\frac N{D^2},\frac1{2\|\alpha d^2\|}\right),
\end{equation}
so that the right-hand side of \eqref{eq:S-dyadic} is at most
$\sum_D\mathcal Y(D)$, because $N/d^2\le N/D^2$ for $d>D$.  We distinguish
four ranges of $D$.  Since $Q^2\le N$, one has $Q/64<N/Q$, so the ranges
$D>Q$, $D^2<Q/64$, $Q/64\le D^2<N/Q$, and $N/Q\le D^2\le Q^2$ are
exhaustive.

\medskip\noindent\emph{Case 1: $D>Q$.}
The sum $\mathcal Y(D)$ has at most $D$ terms, each at most $N/D^2$, so
$\mathcal Y(D)\le N/D$, and summing over the dyadic values $D>Q$,
\begin{equation}\label{eq:case1}
  \sum_{D>Q}\mathcal Y(D)\ll\frac NQ.
\end{equation}

The next two cases are treated as in \cite[p.~595--596]{BP}.

\medskip\noindent\emph{Case 2: $D^2<Q/64$.}
For $D<d\le2D$, \eqref{eq:minor-approx} and $Q^2\le N$ give
\begin{equation}\label{eq:small-beta}
  |\beta|d^2\le\frac{4QD^2}{rN}<\frac{Q^2}{16rN}\le\frac1{16r}.
\end{equation}
The residues $bd^2\pmod r$, $D<d\le2D$, are distinct and nonzero: if
$d_1\ne d_2$ then $0<|d_1^2-d_2^2|<4D^2<Q/16<r$, so $r\nmid d_1^2-d_2^2$,
and $0<d^2<r$ with $(b,r)=1$ shows $r\nmid bd^2$.  Write the least
absolute residue of $bd^2$ modulo $r$ as $\pm j$ with $1\le j\le r/2$.
By \eqref{eq:small-beta},
\[
  \|\alpha d^2\|\ge\frac jr-\frac1{16r}\ge\frac{15j}{16r},
\]
and each $j$ arises from at most two values of $d$.  Hence, using
$r\le N/Q$ and $\log2r\le L$,
\begin{equation}\label{eq:case2}
  \mathcal Y(D)\ll r\sum_{1\le j\le r/2}\frac1j\ll rL\le\frac NQL.
\end{equation}

\medskip\noindent\emph{Case 3: $Q/64\le D^2<N/Q$.}
Since the squares $d^2$, $D<d\le2D$, are distinct integers in
$(D^2,4D^2]$, and $|\beta|\le r^{-2}$ by $r\le N/Q$,
Lemma~\ref{lem:spacing} with $H=4D^2$ and $K=N/D^2$ gives
\begin{equation}\label{eq:case3}
  \mathcal Y(D)\le\sum_{1\le k\le4D^2}\min\left(\frac N{D^2},\frac1{2\|\alpha k\|}\right)
  \ll\left(\frac Nr+D^2+\frac N{D^2}+r\right)L\ll\frac NQL,
\end{equation}
where the last step uses $Q<r\le N/Q$, $D^2<N/Q$, and
$N/D^2\le64N/Q$.

There are $O(L)$ dyadic values of $D$ in Cases~2 and~3 together, so their
total contribution is $O(NQ^{-1}L^2)$.

\medskip\noindent\emph{Case 4: $N/Q\le D^2\le Q^2$.}
Here $2D\le2Q<N^{1/2}$, so $\mathcal Y(D)$ is the sum on the left of
\eqref{eq:sieve-bound}, with $M=N/D^2$.  Put
\[
  P:=\max\left(C_0L,\frac{D^2}{N^{1/2}}\right).
\]
Since $N\ge N_0$ we have $C_0L\le N^{1/2}$, and
$D^2/N^{1/2}\le Q^2/N^{1/2}\le N^{1/2}$, so \eqref{eq:P-range} holds.
All hypotheses of Proposition~\ref{prop:sieve} are satisfied, and
\begin{equation}\label{eq:case4-start}
  \mathcal Y(D)\ll L^4\left(\frac N{Pr}+\frac rP+\frac{D^2}P+\frac{NP}{D^2}\right).
\end{equation}
We sum \eqref{eq:case4-start} over the dyadic $D$ in Case~4, according
to which term defines $P$.

Let $\mathcal D_0$ be the set of dyadic $D$ in Case~4 with
$D^2\le C_0LN^{1/2}$, so that $P=C_0L$.  If $\mathcal D_0$ is nonempty,
the ratio of the largest to the smallest value of $D^2$ in it is at most
$C_0LN^{1/2}/(N/Q)=C_0LQ/N^{1/2}\le C_0L/64$, so
$|\mathcal D_0|\ll\log 2L$, while
$\sum_{D\in\mathcal D_0}D^2\ll C_0LN^{1/2}$ and
$\sum_{D\in\mathcal D_0}D^{-2}\ll Q/N$ by summing geometric progressions.
Therefore
\begin{equation}\label{eq:case4-D0}
  \sum_{D\in\mathcal D_0}\mathcal Y(D)
  \ll L^4\bigg(\frac{\log 2L}{C_0L}\Big(\frac Nr+r\Big)+N^{1/2}+C_0LQ\bigg)
  \ll\frac NQL^5,
\end{equation}
since $N/r<N/Q$, $r\le N/Q$, $N^{1/2}\le N/Q$, and $Q\le N/Q$.

Let $\mathcal D_1$ be the set of the remaining dyadic $D$ in Case~4, for
which $P=D^2/N^{1/2}$.  Then
\[
  \frac N{Pr}=\frac{N^{3/2}}{rD^2},\qquad
  \frac rP=\frac{rN^{1/2}}{D^2},\qquad
  \frac{D^2}P=\frac{NP}{D^2}=N^{1/2}.
\]
Since $D^2\ge N/Q$ throughout Case~4, $\sum_{D\in\mathcal D_1}D^{-2}\ll Q/N$,
and $|\mathcal D_1|\ll L$.  Hence
\begin{equation}\label{eq:case4-D1}
  \sum_{D\in\mathcal D_1}\mathcal Y(D)
  \ll L^4\Big(\frac{N^{1/2}Q}r+\frac{rQ}{N^{1/2}}+N^{1/2}L\Big)
  \ll L^5N^{1/2}\le\frac NQL^5,
\end{equation}
because $r>Q$ and $r\le N/Q$ make the first two terms at most $N^{1/2}$.

Adding \eqref{eq:case1}, the $O(L)$ contributions \eqref{eq:case2} and
\eqref{eq:case3}, and \eqref{eq:case4-D0}--\eqref{eq:case4-D1} proves
\eqref{eq:minor-prop}.
\end{proof}

\begin{proof}[Proof of Theorem~\ref{thm:minor}]
If $Q<128$, the trivial estimate $|S(\alpha)|\le N$ gives
\eqref{eq:minor-bound}.  Let $Q\ge128$ and put $Q_0=Q/64$, so that
$2\le Q_0\le N^{1/2}/64$.  Let $\alpha\in\Minor(Q)$.  By
Lemma~\ref{lem:Dirichlet} with $H=N/Q_0$, there are coprime $b,r$ with
$1\le r\le N/Q_0$ and $|r\alpha-b|\le Q_0/N$.  If $r\le Q_0$, then
$\alpha\in\Major(Q_0)\subseteq\Major(Q)$, which is excluded; hence
$r>Q_0$.  Writing $\alpha=b/r+\beta$ we have $|\beta|\le Q_0/(rN)$, so
\eqref{eq:minor-approx} holds with $Q_0$ in place of $Q$, and
Proposition~\ref{prop:minor} gives
$|S(\alpha)|\ll NQ_0^{-1}L^5\ll NQ^{-1}L^5$.
\end{proof}

\section{Proof of Theorem~\ref{thm:main}}\label{sec:rational}

Let $N\ge2$, and let $a,q,\alpha$ satisfy \eqref{eq:main-hyp}.  If $q>N$
the right-hand side of \eqref{eq:main-bound} exceeds the trivial estimate
$|S(\alpha)|\le N$, so let $1\le q\le N$ and put
\begin{equation}\label{eq:Q-choice}
  Q:=\frac1{64}\min\left(q,\frac Nq\right).
\end{equation}
If $Q<2$, then $\max(q,N/q)>N/128$, and \eqref{eq:main-bound} again
follows from the trivial estimate, since $L^5\gg1$.  So let $Q\ge2$; then
$Q\le N^{1/2}/64$ because $\min(q,N/q)\le N^{1/2}$.

By Lemma~\ref{lem:Dirichlet} with $H=N/Q$ there are coprime $b,r$ with
$1\le r\le N/Q$ and $|r\alpha-b|\le Q/N$.  We claim that $r>Q$.  Indeed,
if $r\le Q$, then by \eqref{eq:main-hyp}, \eqref{eq:Q-choice}, and
$r\le Q\le q/64$,
\[
  |ar-bq|=|r(a-q\alpha)+q(r\alpha-b)|
  \le\frac rq+\frac{qQ}N\le\frac1{64}+\frac1{64}<1,
\]
so $ar=bq$.  Since $a/q$ and $b/r$ are both in lowest terms, this forces
$r=q$, contradicting $r\le q/64$.  Hence $r>Q$, and writing
$\alpha=b/r+\beta$ we have $|\beta|\le Q/(rN)$.  Thus
\eqref{eq:minor-approx} holds, and Proposition~\ref{prop:minor} gives
\[
  |S(\alpha)|\ll\frac NQL^5=64\max\left(\frac Nq,q\right)L^5
  \le64\left(\frac Nq+q\right)L^5,
\]
which is \eqref{eq:main-bound}.  Note that nothing in
Sections~\ref{sec:sieve} and~\ref{sec:minor} requires $\alpha\in[0,1]$, so
the theorem holds for all real $\alpha$.
\qed

\section{Sums of three squarefree numbers}\label{sec:three}

We describe what Theorem~\ref{thm:main} yields for $r_3(N)$ within the
argument of Br\"udern and Perelli~\cite[Sections~4--5]{BP}, and where a
factor $N^{\eps}$ remains.  Throughout this section $N$ is a positive
integer, and we use the notation of \cite{BP}.  Let
\begin{equation}\label{eq:S-tilde}
  \tilde S(\alpha):=\sum_{n\ge1}e^{-n/N}\mu^2(n)\e(\alpha n),
\end{equation}
so that, by orthogonality, $\int_0^1\tilde S(\alpha)^\nu\e(-\alpha N)\,d\alpha=e^{-1}r_\nu(N)$
\cite[(34)]{BP}.  Let $G$ be the multiplicative function with
$G(p)=G(p^2)=(1-p^2)^{-1}$ and $G(p^l)=0$ for $l\ge3$
\cite[(10)]{BP}.  If $q$ is cube-free, write $q=q_1q_2^2$ with $q_1,q_2$
squarefree and coprime; then
\begin{equation}\label{eq:G-bound}
  |G(q)|=\prod_{p\mid q_1q_2}\frac1{p^2-1}
  \le\zeta(2)(q_1q_2)^{-2}\le\frac{\zeta(2)}q,
\end{equation}
since $(q_1q_2)^2=qq_1\ge q$; and $G(q)=0$ otherwise.  Br\"udern and
Perelli cover $[N^{-1/2},1+N^{-1/2}]$ by disjoint intervals
$\mathfrak K(q,a)$ with $1\le q\le N^{1/2}$, $(a,q)=1$, each contained in
$\{\alpha:|q\alpha-a|\le N^{-1/2}\}$, and define on $\mathfrak K(q,a)$
\begin{equation}\label{eq:S-star}
  \tilde S^*(\alpha):=\frac6{\pi^2}G(q)\left(\frac1N-2\pi i\left(\alpha-\frac aq\right)\right)^{-1},
  \qquad \Xi(\alpha):=\tilde S(\alpha)-\tilde S^*(\alpha),
\end{equation}
both extended to $\R$ with period $1$ \cite[(35), (36)]{BP}.  Since
$|N^{-1}-2\pi i\beta|\ge\frac12(N^{-1}+2\pi|\beta|)$, \eqref{eq:G-bound}
gives, for $\alpha\in\mathfrak K(q,a)$ and $\beta=\alpha-a/q$,
\begin{equation}\label{eq:S-star-bound}
  |\tilde S^*(\alpha)|\ll\frac Nq\left(1+N|\beta|\right)^{-1}.
\end{equation}
This is the estimate at the top of \cite[p.~607]{BP}, with the factor
$q^{\eps}$ removed by \eqref{eq:G-bound}.

The sets $\Major(Q)$ and $\Minor(Q)$ of \eqref{eq:major-arcs} differ from
the sets $\mathfrak M(Q)$ and $\mathfrak m(Q)$ of \cite{BP} only by a
translation modulo $1$ and by finitely many points, so integrals of
$1$-periodic functions over them coincide, and we use our notation.  The
following statement replaces \cite[Lemma~5 and (37)]{BP}, which are
restricted to $Q\le N^{3/7}$.

\begin{corollary}\label{cor:Xi}
Let $N\ge2$ and $1\le Q\le\frac12N^{1/2}$.  Then
\begin{equation}\label{eq:Xi-sup}
  \sup_{\alpha\in\Minor(Q)}|\tilde S(\alpha)|\ll\frac NQL^5,
  \qquad
  \sup_{\alpha\in\Minor(Q)}|\Xi(\alpha)|\ll\frac NQL^5.
\end{equation}
\end{corollary}

\begin{proof}
If $Q<128$, both bounds follow from $|\tilde S(\alpha)|\le N+1$ and
\eqref{eq:S-star-bound}.  Let $Q\ge128$ and $\alpha\in\Minor(Q)$.  As in
the proof of Theorem~\ref{thm:minor}, with $Q_0=Q/64$, there are coprime
$b,r$ with $Q_0<r\le N/Q_0$ and $|r\alpha-b|\le Q_0/N\le r^{-1}$, so
Theorem~\ref{thm:main} applies with $a/q=b/r$ and any $x\ge2$ in place of
$N$, and gives $|S(\alpha;x)|\ll(x/r+r)(\log2x)^5$; for $0<x<2$ we have
$|S(\alpha;x)|\le1$.  Partial summation gives
\[
  \tilde S(\alpha)=\frac1N\int_0^\infty S(\alpha;x)\,e^{-x/N}\,dx,
\]
so that
\[
  |\tilde S(\alpha)|\le\frac2N+\frac1N\int_2^\infty|S(\alpha;x)|e^{-x/N}\,dx
  \ll1+\frac1N\int_2^\infty\left(\frac xr+r\right)(\log2x)^5e^{-x/N}\,dx.
\]
Since $\log 2x\le L+x/N$ for $x\ge1$, we have
\[
  \int_2^\infty(\log 2x)^5e^{-x/N}\,dx\ll NL^5,\qquad
  \int_2^\infty x(\log 2x)^5e^{-x/N}\,dx\ll N^2L^5,
\]
whence $|\tilde S(\alpha)|\ll(N/r+r)L^5\ll NQ_0^{-1}L^5\ll NQ^{-1}L^5$.

For $\tilde S^*$, let $\alpha\in\mathfrak K(q,a)$ with $q\le N^{1/2}$ and
$\beta=\alpha-a/q$.  If $q>Q$, then \eqref{eq:S-star-bound} gives
$|\tilde S^*(\alpha)|\ll N/q<N/Q$.  If $q\le Q$, then $\alpha\notin\Major(Q)$
forces $|q\beta|>Q/N$, and \eqref{eq:S-star-bound} gives
$|\tilde S^*(\alpha)|\ll(qN|\beta|)^{-1}N<N/Q$.  The second bound in
\eqref{eq:Xi-sup} follows.
\end{proof}

Corollary~\ref{cor:Xi} sharpens the only ingredient of \cite[Section~5]{BP}
that was restricted to $Q\le N^{3/7}$, and yields, as Br\"udern and
Perelli anticipated \cite[p.~610]{BP}, the same precision for $\nu=3$ as
their Theorem~3 gives for $\nu\ge4$.  This is Tolev's
\cite[Corollary~2]{Tolev}, which also follows from \cite{SP}; we include
the deduction to make precise where the factor $N^{\eps}$ survives.  Let
\begin{equation}\label{eq:singular-series}
  \mathfrak S_3(N):=\prod_{p^2\nmid N}\left(1-\frac1{(1-p^2)^3}\right)
  \prod_{p^2\mid N}\left(1-\frac1{(1-p^2)^2}\right)
\end{equation}
be the singular series \cite[(2)]{BP}.

\begin{corollary}[Br\"udern--Perelli, Schlage-Puchta, Tolev]\label{cor:r3}
Let $N$ be a positive integer, and assume the generalized Riemann
hypothesis for Dirichlet $L$-functions.  Then
\begin{equation}\label{eq:r3}
  r_3(N)=\frac12\left(\frac6{\pi^2}\right)^3\mathfrak S_3(N)N^2
  +O(N^{5/4+\eps}).
\end{equation}
\end{corollary}

\begin{proof}
By \cite[(38), (42), (45)]{BP} and the evaluation of the main term on
\cite[p.~610]{BP},
\[
  e^{-1}r_3(N)=\frac{N^2}{2e}\left(\frac6{\pi^2}\right)^3\mathfrak S_3(N)
  +O(N^{5/4+\eps}+E_5),
  \qquad E_5=\int_0^1|\Xi(\alpha)|^3\,d\alpha,
\]
so it suffices to show $E_5\ll N^{5/4+\eps}$.  Under GRH,
\cite[Lemma~4]{BP} states that
\begin{equation}\label{eq:BP-lemma4}
  \int_{\Major(Q)}|\Xi(\alpha)|^2\,d\alpha\ll N^{\eps-1/2}Q^{5/2}
  \quad\left(1\le Q\le\tfrac12N^{1/2}\right),
  \qquad
  \int_0^1|\Xi(\alpha)|^2\,d\alpha\ll N^{3/4+\eps}.
\end{equation}
Put $Q_1=\frac12N^{1/2}$.  By Corollary~\ref{cor:Xi} and
\eqref{eq:BP-lemma4},
\[
  \int_{\Minor(Q_1)}|\Xi|^3
  \le\sup_{\Minor(Q_1)}|\Xi|\int_0^1|\Xi|^2
  \ll N^{1/2}L^5\cdot N^{3/4+\eps}\ll N^{5/4+\eps}.
\]
The set $\Major(Q_1)$ is covered by $\Major(1)$ and by $O(L)$ sets
$\Major(2Q)\setminus\Major(Q)$ with $1\le Q\le\frac12Q_1$.  Since
$\Major(2Q)\setminus\Major(Q)\subseteq\Minor(Q)$, Corollary~\ref{cor:Xi}
and \eqref{eq:BP-lemma4} give
\[
  \int_{\Major(2Q)\setminus\Major(Q)}|\Xi|^3
  \le\sup_{\Minor(Q)}|\Xi|\int_{\Major(2Q)}|\Xi|^2
  \ll\frac NQL^5\cdot N^{\eps-1/2}Q^{5/2}
  =N^{1/2+\eps}Q^{3/2}L^5\ll N^{5/4+\eps},
\]
while, by \eqref{eq:BP-lemma4} with $Q=1$ and the trivial bound
$|\Xi(\alpha)|\ll N$, which follows from $|\tilde S(\alpha)|\le N+1$ and
\eqref{eq:S-star-bound},
\[
  \int_{\Major(1)}|\Xi|^3\le\sup_{\Major(1)}|\Xi|\int_{\Major(1)}|\Xi|^2
  \ll N\cdot N^{\eps-1/2}=N^{1/2+\eps}.
\]
Summing over the $O(L)$ values of $Q$, with $\eps/2$ in place of $\eps$ in
\eqref{eq:BP-lemma4} and $L^6\ll_{\eps}N^{\eps/2}$, gives
$E_5\ll N^{5/4+\eps}$.
\end{proof}

The factor $N^{\eps}$ in \eqref{eq:r3} no longer comes from the pointwise
minor-arc estimate.  It enters through \eqref{eq:BP-lemma4}, which is used
both in the bound for $E_5$ above and in the estimate
$E_4\ll N^{5/4+\eps}$ of \cite[(42)]{BP}, where
$E_4=\int_0^1|\tilde S^*(\alpha)|^2|\Xi(\alpha)|\,d\alpha$ is the other
error term in \cite[(38)]{BP}.  In the proof of
\cite[Lemma~3]{BP} the mean square of $\Xi$ is expressed through Mellin
integrals of $L(s,\chi)/L(2s,\chi^2)$ on the line $\Re s=\frac14+\eps$,
GRH is used to bound the integrand there, and the dependence on $\eps$ is
recorded as $N^{\eps}$.  A version of \eqref{eq:BP-lemma4} with a power of
$\log N$ in place of $N^{\eps}$ would require a separate analysis of this
conditional mean-square estimate, which the present paper does not
provide.  Consequently, improving the pointwise estimate of
Theorem~\ref{thm:minor} alone cannot remove the factor $N^{\eps}$ from
\eqref{eq:r3}.

\section*{Acknowledgements}

The conception of the paper, the strategy of the proofs, and
the first complete proof are due to the authors.  The aim was to replace
the factor $N^{\eps}$ in the estimates of Schlage-Puchta and Tolev by
$(\log N)^A$ for some fixed $A>0$; the first version had $A=12$, with no
attempt at optimality.  The manuscript was then subjected to an adversarial proof audit with
the language model ChatGPT Sol~5.6 (OpenAI).  In the course of that audit
the model reduced the exponent to $A=6$, by sharpening the count of
representations $\ell=hpp'$ to $R(\ell)\le(\log\ell/\log P)^2$, whose
factor $(\log P)^{-2}$ cancels the $(\log P)^2$ arising from $A^{-2}$ in
\eqref{eq:A-lower}; a subsequent joint effort of the authors and the model
brought it to $A=5$, by summing the four terms of \eqref{eq:sieve-bound}
over the dyadic blocks separately in the two regimes $P=C_0L$ and
$P=D^2/N^{1/2}$ instead of bounding every block uniformly. The authors
did not pursue a further reduction of the exponent below $A=5$, since the
point of the paper -- that the factor $N^{\eps}$ can be replaced by a fixed
power of $\log N$ -- was already established.  The language
model Claude (Anthropic) was used for a further independent audit and to
improve the exposition of the final text.  All mathematical content was
verified by the authors, who take full responsibility for it.

\end{document}